\documentclass[11pt,reqno]{amsart}

\usepackage{amsmath,amssymb,amscd,graphicx, amsthm,
  enumerate, hyperref}

\usepackage{mathrsfs}
\usepackage{mathtools}
\usepackage[all]{xy}
\usepackage{amsmath}
\usepackage{color,xcolor}
\definecolor{darkred}{rgb}{1,0,0} 
\definecolor{darkgreen}{rgb}{0,0.8,0}
\definecolor{darkblue}{rgb}{0,0,1}

\hypersetup{colorlinks,
linkcolor=darkblue,
filecolor=darkgreen,
urlcolor=darkred,
citecolor=darkgreen}

\let\inf\relax \DeclareMathOperator*\inf{\vphantom{p}inf}

\numberwithin{equation}{section}

\theoremstyle{plain}

\theoremstyle{plain}
\newtheorem{theorem}{Theorem}
\numberwithin{theorem}{section}
\newtheorem{proposition}[theorem]{Proposition}

\theoremstyle{definition}

\theoremstyle{definition}
\newtheorem{remark}[theorem]{Remark}

\newtheorem*{theorem*}{Corollary of Conjecture 4.1}

\newcommand{\Real}{\mathrm{Re}} 
\newcommand{\genv}{\mathrm{gen.val}}

\DeclareMathOperator{\genva}{gen.val}

\title{A remark on two notions of order of contact}

\author{Martino Fassina}

\address{Department of Mathematics, University of Illinois, 1409 W Green
Street, Urbana, IL 61801, USA}
\email{fassina2@illinois.edu}

\begin{document}

\begin{abstract}

We recall two measurements of the order of contact of an ideal in the ring of germs of holomorphic functions at a point and we provide a class of examples in which they differ.

\end{abstract}
\maketitle

\section*{Introduction}

The purpose of this note is to provide examples where two measurements of the order of singularity of an ideal $I$ in the ring of germs of holomorphic functions at a point differ. 
These measurements arise when defining the order of contact of $q$-dimensional complex analytic varieties with a real hypersurface in $\mathbb{C}^n$ (\cite{Da82},\cite{Ca87}). 
The work in \cite{Da93} shows how to reduce such questions about the local geometry of a real hypersurface to questions about ideals in the ring of germs of holomorphic functions at a point. 
We therefore carry out our work in the holomorphic setting, and in the last part of the paper we indicate the consequences in the hypersurface case.

The measurements we compare are known in the literature as \textit{D'Angelo $q$-type} and \textit{Catlin $q$-type}; for an ideal $I$, we denote them respectively by ${\bf T}_q(I)$ and $D_q(I)$ 
(see Section \ref{Section 2} and Section \ref{Section 3} for precise definitions). In the important case $q=1$, we have ${\bf T}_1(I)=D_1(I)$. For $q>1$, Catlin expressed in \cite{Ca87} the hope that
 the two numbers are also equal. In \cite{DaKo99}, D'Angelo and Kohn noted that ${\bf T}_q$ and $D_q$ are simultaneously finite, and more recently, Brinzanescu and Nicoara obtained in \cite{BN15} some 
inequalities relating the two quantities. However, they were not able to produce an example where the two numbers differ. 

In this paper, we exhibit a class of ideals $I$ for which ${\bf T}_q(I)\neq D_q(I)$ (Theorem \ref{theo}). We also show that, in any dimension greater than $2$, the difference $D_q(I)-{\bf T}_q(I)$ can be arbitrarily large (Proposition \ref{distant}).

In Section \ref{Section 4} we derive similar statements in the case of type conditions for real hypersurfaces in $\mathbb{C}^n$.

The crucial point in understanding the difference between the two measurements turns out to be the distinction between an infimum and a generic value (Proposition \ref{ex1}).
In this respect, we point out an error in the work of Brinzanescu and Nicoara, where the generic value is assumed to be always equal to the infimum. The inequalities they obtained in \cite{BN15} hold 
if the quantity ${\bf T}_q$ is replaced by a new invariant $\beta_q$ (see Section \ref{Section 2}) defined in terms of the generic value. 

In the last part of the paper we describe inequalities between the measurements ${\bf T}_q$ and $D_q$. These inequalities are closely related to the inequalities from \cite{Da82} about how ${\bf T}_1$ behaves with respect to parameters.

\section{D'Angelo q-type}\label{Section 2}
Let $\mathcal{O}_n$ denote the local ring of germs of holomorphic functions at $0\in\mathbb{C}^n$. We write $(V,0)$ for the germ at $0$ of a complex variety in $\mathbb{C}^n$. In particular, $(\mathbb{C},0)$ denotes the germ at $0$ of $\mathbb{C}$. We denote by $\Gamma$ the set of non-constant germs of holomorphic functions $z\colon(\mathbb{C},0)\rightarrow(\mathbb{C}^n,0)$. Each such curve is therefore defined in an open neighborhood $U$ of $0$ (depending on $z$) in $\mathbb{C}$. We denote by $v(z)$ the order of vanishing of $z$ at $0$; that is, $v(z)$ is the unique integer $m$ for which $z(t)=t^mu(t)$ for the germ of a holomorphic map $u$ with $u(0)\neq 0$. In \cite{Da82}, D'Angelo introduced, for an ideal $I$ in $\mathcal{O}_n$, the invariant 
\begin{equation}\label{type1}
\textbf{T}_1^n(I)=\sup_{z\in\Gamma}\,\,\inf_{g\in I}\frac{v(g\circ z)}{v(z)},
\end{equation} 
called the {\em first order of contact of $I$} or the {\em 1-type of $I$}.
We drop the superscript $n$ on $\textbf{T}_1^n$ when there is no ambiguity. As will become clear later, sometimes it is useful to keep track of the dimension of the ambient space, especially when we want to exploit the next remark.

\begin{remark}\label{ubi}
For integers $m\leq n$, let $(z_1,\dots,z_m)$ be the ideal generated in $\mathcal{O}_n$ by the first $m$ coordinate functions. If $I\subset (z_1,\dots,z_m)$ is an ideal, then ${\bf T}_1^m(I)={\bf T}_1^n(I,z_{m+1},\dots,z_{n})$.
\end{remark}
D'Angelo also defined, for an ideal $I$ in $\mathcal{O}_n$ and $q\in\{1,\dots,n\}$, the {\em q-type of $I$} as
\begin{equation}\label{deltaq}
{\bf T}_q(I)= \inf_{\{w_1\dots,w_{q-1}\}} \,\,{\bf T}_1(I,w_1\dots,w_{q-1}).
\end{equation}
Here the infimum is taken over all choices of linear functions $\{w_1,\dots,w_{q-1}\}$, and $(I,w_1,\dots,w_{q-1})$ denotes the ideal generated by $I$ and $w_1,\dots,w_{q-1}$.

\begin{remark}\label{genericity}
Let $I$ be an ideal in $\mathcal{O}_n$ and let $G^{n-q+1}$ denote the Grassmannian of $(n-q+1)$-dimensional vector subspaces of $\mathbb{C}^n$. There exists $\beta\in \mathbb{R}\cup\{\infty\}$ and a non-empty open subset $W$ of $G^{n-q+1}$ such that if $S\in W$ and $S$ is defined by $w_j=0$ for $j=1,\dots,q-1$, then ${\bf T}_1(I,w_1,\dots,w_{q-1})=\beta$. If the dimension of  $(V(I),0)$ is greater than $q-1$, 
the assertion follows from taking $W=G^{n-q+1}$ and $\beta=\infty$. If $(V(I),0)$ is of dimension at most $q-1$, we can take $W$ to be the set of all hyperplanes transversal to $(V(I),0)$. 
We recall that every non-empty open subset of $G^{n-q+1}$ in the Zariski topology is also dense. 
\end{remark}

The number $\beta$ whose existence is noted in Remark \ref{genericity} will be called the \textit{generic value} of $\textbf{T}_1(I,w_1,\dots,w_{q-1})$. We can thus define, 
for an ideal $I$ in $\mathcal{O}_n$ and $q\in\{1,\dots,n\}$, the invariant 
$$\beta_q(I)=\underset{\{w_1,\dots,w_{q-1}\}}\genv{\bf T}_1(I,w_1,\dots,w_{q-1}).$$
As we show below, the generic value $\beta_q$ need not equal the $q$-type $\textbf{T}_q$. 
The result is inspired by \cite{Da80}, where D'Angelo showed that $\textbf{T}_1$ is not an upper semicontinuous function of parameters.
\begin{proposition}\label{ex1}
There exist ideals $I$ in $\mathcal{O}_3$ for which ${\bf T}_2(I)<\beta_2(I)$.
\end{proposition}
\begin{proof}
Let $I=(z_1^3-z_3z_2,z_2^2)\subset\mathcal{O}_3$. We show that ${\bf T}_2(I)=3$ and $\beta_2(I)=4$.
For $a,b,c\in\mathbb{C}$, consider the quantity $\textbf{T}_1(I,az_1+bz_2+cz_3)$. \\
Assume first that $c\neq 0$. Hence, without loss of generality, $c=1$. Since the type ${\bf T}_1$ is invariant under a local biholomorphic change of coordinates preserving $0$, 
we can apply the change of variables 
$w_1=z_1$, $w_2=z_2$, $w_3=z_3+az_1+bz_2$. We thus obtain 
\begin{equation*}
 \begin{split}                                                                                                                                                                                                                                                                                                               
{\bf T}_1(w_1^3-(w_3-aw_1-bw_2)w_2,w_2^2,w_3)&={\bf T}_1(w_1^3+aw_1w_2+bw_2^2,w_2^2,w_3)\\&={\bf T}_1(w_1^3+aw_1w_2,w_2^2,w_3).
\end{split}
\end{equation*}
Remark \ref{ubi} implies $\textbf{T}^3_1(w_1^3+aw_1w_2,w_2^2,w_3)=\textbf{T}^2_1(w_1^3+aw_1w_2,w_2^2)$. 
We now distinguish two cases. If $a\neq 0$, then $\textbf{T}^2_1(w_1^3+aw_1w_2,w_2^2)=4$, as one can see by 
considering the curve $\gamma_{a}(t)=(t,-t^2/a)$. If $a=0$, then $\textbf{T}^2_1(w_1^3+aw_1w_2,w_2^2)=3$. \\
When $c=0$, we have $\textbf{T}_1(I,az_1+bz_2+cz_3)=\infty$.\\ In conclusion,
\begin{equation*}
{\bf T}_2(I)=\underset{\{a,b,c\}}\inf\,\,\textbf{T}_1(I,az_1+bz_2+cz_3)=3.
\end{equation*}
\begin{equation*}
\beta_2(I)=\underset{\{a,b,c\}}\genva\,\,\textbf{T}_1(I,az_1+bz_2+cz_3)=4.
\end{equation*}
\end{proof}

Proposition \ref{ex1} provides a counterexample to the following statement, which appears as Corollary 2.11 in \cite{BN15}.
\medskip

\textit {If $I$ is a proper ideal in $\mathcal{O}_n$, the infimum in the definition of ${\bf T}_q(I)$ is achieved and equal to the generic value,
$${\bf T}_q(I)=\underset{\{w_1,\dots,w_{q-1}\}}\genv{\bf T}_1(I,w_1,\dots,w_{q-1}).$$ }

\section{Catlin q-type}\label{Section 3}
We recall from \cite{Ca87} the definition of Catlin $q$-type. (See also Section 3 of \cite{BN15}). 
Let $(V^q,0)$ be the germ of a complex $q$-dimensional variety in $\mathbb{C}^n$ and let $G^{n-q+1}$ denote the Grassmannian of $(n-q+1)$-dimensional vector subspaces of $\mathbb{C}^n$. 
By Proposition 3.1 in \cite{Ca87}, there exists a non-empty open subset $W$ of $G^{n-q+1}$ such that for all $S\in W$ the intersection $V^q\cap S$ at $0$ consists of finitely many irreducible $1$-dimensional components, whose germs at $0$ are denoted by $\gamma_S^k\colon (\mathbb{C},0)\rightarrow (\mathbb{C}^n,0)$, for $k=1,\dots,P$. For every $g\in\mathcal{O}_n$ consider the quantity
\begin{equation}\label{quant}
 \max_{k=1,\dots,P}\frac{v(g\circ \gamma_{S}^k)}{v(\gamma_S^k)}.
 \end{equation}
In Proposition 3.1 of \cite{Ca87}, Catlin showed that \eqref{quant} gives the same value for all $S$ in a non-empty open subset of $W$ (depending on the variety $V$). Therefore, for a given variety $V$, the value \eqref{quant} is generic over the choice of $S$. 
We can thus define, for an ideal $I$ in $\mathcal{O}_n$ and a positive integer $q$, the \textit{Catlin $q$-type of $I$} to be the number
$$D_q(I)=\sup_{V^{q}}\,\,\inf_{g\in I}\,\,\underset{S\in G^{n-q+1}}\genva\,\, \max_{k=1,\dots,P}\frac{v(g\circ \gamma_{S}^k)}{v(\gamma_S^k)}.  $$

We next give an example of an ideal $I$ in $\mathcal{O}_3$ for which $\textbf{T}_2(I)\neq D_2(I)$. The author knows of no such example in the literature. 
\begin{proposition}\label{ex}
 There exist ideals $I$ in $\mathcal{O}_3$ for which $\textbf{T}_2(I)<D_2(I)$.
\end{proposition}
\begin{proof}
Let $I=(z_1^3-z_3z_2,z_2^2)\subset \mathcal{O}_3$. We show that the strict inequality $\textbf{T}_2(I)<D_2(I)$ holds. We have already noted in the proof of Proposition \ref{ex1} that
$\textbf{T}_2(I)=3$. We now argue that $D_2(I)\geq 4$. Consider the $2$-dimensional variety $V$ in $\mathbb{C}^3$ defined by the equation $z_1^3-z_3z_2=0$. 
For each $2$-dimensional plane $S_{a,b}$ in $\mathbb{C}^3$ through 0 defined by $z_3+az_1+bz_2=0$ with $a\neq 0, b\neq 0$, we prove that there exists an irreducible 1-dimensional 
component of $V\cap S_{a,b}$, with (minimal) parameterization $\gamma_{a,b}$, such that 
\begin{equation}\label{incatllemma}
\forall g\in I \quad \frac{v(g\circ\gamma_{a,b})}{v(\gamma_{a,b})}\geq 4.
\end{equation}
Replacing $z_3=-az_1-bz_2$ into $z_1^3-z_3z_2=0$ and solving the second equation for $z_2$, we see that the $1$-dimensional irreducible components of $V\cap S_{a,b}$ are defined 
by the equations
$$z_2=\frac{-az_1}{2b}\pm\frac{az_1}{2b}\left(1-\frac{4bz_1}{a^2}\right)^{\frac{1}{2}},\quad\quad z_3=-az_1-bz_2.$$
It suffices to consider the plus sign. The resulting curve has a local parameterization at 0 given by
$$\gamma_{a,b}(t)=\left(t, -\frac{1}{a}t^2+\dots, -at +\frac{b}{a}t^2+\dots\right),$$
where dots denote higher order terms in $t$. \\For $g\in I$, write $g=h\cdot(z_1^3-z_3z_2)+k\cdot(z_2^2)$ with $h,k\in\mathcal{O}_3$. Since $\gamma^*_{a,b}(z_1^3-z_3z_2)\equiv 0$ and $v(\gamma_{a,b})=1$, 
it follows, for all $g\in I$, that 
$$\frac{v(g\circ\gamma_{a,b})}{v(\gamma_{a,b})}\geq v(z_2^2\circ\gamma_{a,b})=4.$$
Therefore \eqref{incatllemma} holds.
Thus ${\bf T}_2(I)=3<4\leq D_2(I).$
\end{proof}

Already when $q=2, n=3$, we can find examples of ideals for which the difference between the two quantities is arbitrarily large.

\begin{proposition}\label{2,3}
Let $k\in\mathbb{N}$. There exist ideals $I$ in $\mathcal{O}_3$ for which $$D_2(I)-{\bf T}_2(I)>k.$$
\end{proposition}

\begin{proof}
Given $k$, choose $m\in\mathbb{N}$ such that $m^2-2m>k$ and consider the ideal $I=(z_1^m-z_3z_2, z_2^m)$. 
We apply the same argument as in Proposition \ref{ex}. It is immediately verified that ${\bf T}_2(I)\le {\bf T}_1(I,z_3)=m$. We now prove that $D_2(I)\geq m(m-1)$, from which the conclusion will follow. 
Consider the $2$-dimensional variety $V$ in $\mathbb{C}^3$ defined by the equation $z_1^m-z_3z_2=0$. The same computations as in Proposition \ref{ex} show, for every $2$-dimensional plane $S_{a,b}$ in $\mathbb{C}^3$ through 0 defined by an equation $z_3+az_1+bz_2=0$ with $a\neq 0, b\neq 0$, that there exists an irreducible 1-dimensional component of $V\cap S_{a,b}$ whose local parameterization at 0 is given by
$$\gamma_{a,b}(t)=\left(t, -\frac{1}{a}t^{m-1}+\dots, -at +\frac{b}{a}t^{m-1}+\dots\right).$$
Once again, dots denote higher order terms in $t$.  \\For $g\in I$, write $g=h\cdot(z_1^m-z_3z_2)+k\cdot(z_2^m)$ with $h,k\in\mathcal{O}_3$. Since $\gamma^*_{a,b}(z_1^m-z_3z_2)\equiv 0$ and $v(\gamma_{a,b})=1$, 
it follows, for all $g\in I$, that \begin{equation*}
 \frac{v(g\circ\gamma_{a,b})}{v(\gamma_{a,b})}\geq v(z_2^m\circ\gamma_{a,b})=m(m-1).
\end{equation*} 
Hence $D_2(I)\geq m(m-1)=m^2-2m+m>k+{\bf T}_2(I).$
\end{proof}

In the spirit of Remark \ref{ubi}, we can adjoin new variables to generalize Proposition \ref{2,3} to arbitrary dimension. We give the general statement, omitting the simple proof.

\begin{proposition}\label{distant}
Let $q,k,n\in\mathbb{N}$, $q\geq 2$, $n\geq q+1$. There exist ideals $I$ in $\mathcal{O}_n$ for which $D_q(I)-{\bf T}_q(I)>k$.
\end{proposition}

The next theorem generalizes the work above, providing a wider class of examples where the two types differ. Once again, the distinction between the infimum and the generic value plays a major role.
\begin{theorem}\label{theo}
Let $I=\big{(}f(z_1)-z_2z_3,g(z_2)\big{)}$ be an ideal in $\mathcal{O}_3$ where $f$ and $g$ are holomorphic functions with $v(f)\geq 3$ and $v(g)\geq 2$. Then $$D_2(I)\geq \beta_2(I)>{\bf T}_2(I).$$
\end{theorem}
\begin{proof}
First note that 
\begin{equation}\label{pio}
\begin{split}
{\bf T}_2(I)\leq {\bf T}^3_1(f(z_1)-z_2z_3,g(z_2),z_3)&={\bf T}^2_1(f(z_1), g(z_2))\\&=\max\,\{v(f),v(g)\}.
\end{split}
\end{equation}
We next compute the quantity  
$$\beta_2(I)=\underset{\{a,b\}}\genva\,\,{\bf T}_1\big{(}f(z_1)+az_1z_2+bz_2^2,g(z_2)\big{)}.$$
We can assume without loss of generality that $a\neq 0$ and $b\neq 0$. It suffices to test the germs of the curves defined by $g(z_2)=0$ or $f(z_1)+az_1z_2+bz_2^2=0$. The equation $g(z_2)=0$ defines at $0$ the germ of the curve $\sigma(t)=(t,0)$, and $$v(\sigma^*(f(z_1)+az_1z_2+bz_2^2))=v(f).$$ The function $f(z_1)+az_1z_2+bz_2^2$ is a quadratic in $z_2$. One of the components of its zeroset has (minimal) local parameterization at $0$ given by
\begin{equation}\label{gamma}
\gamma_{a,b}(t)=\left(t,-\frac{f(t)}{at}+\dots\right),
\end{equation}
where dots denote higher order terms in $t$.
Note that $$v(g\circ \gamma_{a,b})=v(g)(v(f)-1).$$ We therefore conclude that $$\beta_2(I)=\max\,\{v(f), v(g)(v(f)-1)\}.$$ By the hypothesis on $v(f)$ and $v(g)$, we have $\beta_2(I)=v(g)(v(f)-1)$. Combining with inequality \eqref{pio}, we conclude that $\beta_2(I)>{\bf T}_2(I)$. Now consider the $2$-dimensional variety $V$ defined in $\mathbb{C}^3$ by the equation $f(z_1)-z_2z_3=0$. The computation above shows, for every $2$-dimensional plane $S_{a,b}$ in $\mathbb{C}^3$ through 0 defined by an equation $z_3+az_1+bz_2=0$ with $a\neq 0, b\neq 0$, that there exists an irreducible 1-dimensional component of $V\cap S_{a,b}$ whose local parameterization $\gamma_{a,b}$ at 0 is given by \eqref{gamma}.
For all $h\in I$ we have 
$$\frac{v(h\circ\gamma_{a,b})}{v(\gamma_{a,b})}\geq v(g\circ\gamma_{a,b})=v(g)(v(f)-1).$$ 
It follows that $D_2(I)\geq v(g)(v(f)-1)=\beta_2(I)$. 
\end{proof}

\section{The type of a real hypersurface}\label{Section 4}
Let $C^{\infty}_p$ denote the ring of germs of smooth functions at a point $p\in\mathbb{C}^n$. Let $(M,p)$ be the germ of a smooth real hypersurface in $\mathbb{C}^n$ and let $r_p$ denote a generator of the principal ideal in $C^{\infty}_p$ of functions that vanish on $M$. We write $\Gamma$ for the set of non-constant germs of holomorphic functions $z\colon(\mathbb{C},0)\rightarrow(\mathbb{C}^n,p)$. In \cite{Da82}, D'Angelo defined the maximum order of contact of complex analytic curves with $M$ at $p$ as the number $$\Delta_1(M,p)=\sup_{z\in\Gamma_p}\frac{v(r_p\circ z)}{v(z)}.$$
More generally, when $q\in\{1,\dots,n\}$, the maximum order of contact of $q$-dimensional complex analytic varieties with $M$ at $p$ is defined as 
\begin{equation*}
\Delta_{q}(M,p)=\inf_{S\in G^{n-q+1}_p}\Delta_1(M\cap S,p),
\end{equation*} 
where $G^{n-q+1}_p$ denotes the set of $(n-q+1)$-dimensional complex affine subspaces of $\mathbb{C}^n$ through $p$, and $(M\cap S,p)$ is regarded as the germ of a smooth real hypersurface in $\mathbb{C}^{n-q+1}$.
The invariant $\Delta_q(M,p)$ is called the {\em D'Angelo q-type} of $M$ at $p$.

In the same setting, the \textit{Catlin $q$-type} is defined in \cite{Ca87} as
\begin{equation*}
D_q(M,p)=\sup_{V^{q}}\,\,\underset{S\in G^{n-q+1}_p}\genva\,\, \max_{k=1,\dots,P}\frac{v(r_p\circ \gamma_{S}^k)}{v(\gamma_S^k)}.  
\end{equation*}
Here the supremum is taken over all germs $(V^q,p)$ of $q$-dimensional complex varieties, and for a generic $S\in G^{n-q+1}_p$, we denote by $\gamma_{S}^1,\dots,\gamma_{S}^P$ the germs at $0$ of the 
$1$-dimensional irreducible components of $V^q\cap S$ (see the beginning of Section \ref{Section 3}). 

In \cite{Da93}, D'Angelo made precise the relationship between the $1$-type $\Delta_1$ of the germ of a smooth hypersurface in $\mathbb{C}^n$ and the corresponding notion of $1$-type ${\bf T}_1$ 
for holomorphic ideals in $\mathcal{O}_n$. 
We consider here a situation in which this relationship is particularly simple. Let $M$ be defined locally at $p$ by an equation $r=0$, with
\begin{equation*}
r(z,\bar{z})=\Real (h)+\sum_{j=1}^t|f_j|^2
\end{equation*} 
for some holomorphic functions $h,f_1,\dots,f_t$ with $h(p)=f_j(p)=0$ and $dh(p)\neq 0$. Consider the associated ideal $I(M,p)=(h,f_1,\dots,f_t)$ in $\mathcal{O}_n$. We have (\cite[Theorem 4.7]{Da93})
\begin{equation}\label{pscvx}
\Delta_1(M,p)=2{\bf T}_1(I(M,p)).
\end{equation}

This discussion allows us to translate Propositions \ref{ex1} and \ref{ex} to the hypersurface setting. 
\begin{proposition}
There exist real hypersurfaces $M$ in $\mathbb{C}^4$ for which $$\Delta_2(M,0)< \,\underset{S\in G^{3}_0}\genva \,\,\Delta_1 (M\cap S,0).$$
\end{proposition}
\begin{proof}
Consider the real hypersurface $M$ in $\mathbb{C}^4$ defined by the equation $r(z,\bar{z})=0,$ where $r(z,\bar{z})=\Real (z_4)+|z_1^3-z_3z_2|^2+|z_2|^4$. 
Note that the associated ideal is $I=(z_4,z_1^3-z_3z_2,z_2^2)\subset\mathcal{O}_4$. Combining \eqref{pscvx} with Propositions \ref{ex1} and \ref{ex} we obtain 
$$\Delta_2(M,0)=6<8= \,\underset{S\in G^{3}_0}\genva \,\,\Delta_1 (M\cap S,0).$$
\end{proof}

\begin{proposition}
There exist real hypersurfaces $M$ in $\mathbb{C}^4$ for which $$\Delta_2(M,0)<  D_2(M,0).$$
\end{proposition}
\begin{proof}
Consider the real hypersurface $M$ in $\mathbb{C}^4$ defined by the equation $r(z,\bar{z})=0,$ where $r(z,\bar{z})=\Real (z_4)+|z_1^3-z_3z_2|^2+|z_2|^4$. Let $L$ be the hyperplane in $\mathbb{C}^4$ defined by 
the equation 
$z_4=0$. By Proposition \ref{ex} there exists a $2$-dimensional variety $V\subset L$ such that, for a generic $2$-dimensional plane $S$ in $L$, the intersection $S\cap V$ has a $1$-dimensional component whose germ 
$\gamma_S$ at $0$ satisfies
$$\frac{v(r\circ\gamma_S)}{v(\gamma_S)}\geq 8.$$
Therefore $D_2(M,0)\geq 8>6=\Delta_2(M,0).$
\end{proof}

\section{Inequalities between the two notions of type}
In our treatement so far we have presented several examples where the D'Angelo $q$-type and the Catlin $q$-type differ. We have also seen that this phenomenon involves the distinction between an infimum and a generic value. 
It is therefore natural to ask if there is a perfect correspondence between the two numbers when the D'Angelo $q$-type is replaced by an invariant defined in terms of the generic value. 
The answer was recently given by Brinzanescu and Nicoara in \cite{BN17}. After having been informed of Proposition \ref{ex1} of this paper, they have subsequently corrected their error and proved the following theorem, as well as the analogous result for the hypersurface case.

\begin{theorem}\cite[Theorem 1.5 (i)]{BN17}\label{nico}
 Let $I$ be an ideal in $\mathcal{O}_n$. For $1\leq q\leq n$, we have $D_q(I)=\beta_q(I)$.
\end{theorem}

The inequality $D_q(I)\leq \beta_q(I)$ is not hard. A proof of this fact is already present in the proof of Theorem 1.1 in \cite{BN15}, even though the authors state their result in terms of ${\bf T}_q(I)$, erroneously assuming that ${\bf T}_q(I)=\beta_q(I)$ in general. The reverse inequality $\beta_q(I)\leq D_q(I)$ is more subtle. The author hopes to give in a future paper a different proof of this inequality.

In Section \ref{Section 3} we have seen, for $q\geq 2$, that the difference $D_q(I)-{\bf T}_q(I)$ can be arbitrarily large. We now show how Theorem \ref{nico} implies
 some uniform inequalities between ${\bf T}_q$ and $D_q$. In particular, $D_q$ is bounded above by a power of ${\bf T}_q$ which depends only on $q$ and the dimension $n$ of the ambient space. (A similar statement appears in Proposition 1.4 of \cite{BN17}). 

We first recall an auxiliary result.
For an ideal $I$ in $\mathcal{O}_n$, let $${\bf mult}(I)=\dim_{\mathbb{C}}\mathcal{O}_n/I$$ be the \textit{codimension of $I$}. The following inequality is due to D'Angelo (\cite[Theorem 2.4]{Da93}).

\medskip

\textit{If $I$ contains $q$ independent linear functions, then}
\begin{equation}\label{inequalitymult}
{\bf T}_1(I)\leq {\bf mult}(I)\leq \left({\bf T}_1(I)\right)^{n-q}.
\end{equation}

\begin{proposition}
Let $I$ be an ideal in $\mathcal{O}_n$. Then, for $1\leq q\leq n$,
\begin{equation*}
{\bf T}_q(I)\leq D_q(I)\leq \left({\bf T}_q(I)\right)^{n-q+1}.
\end{equation*}
\end{proposition}
\begin{proof}
The first inequality follows immediately from Theorem \ref{nico} and the obvious fact ${\bf T}_q(I)\leq\beta_q(I)$. 
Let now $\tilde{w}_1,\dots,\tilde{w}_{q-1}$ be linear functions such that ${\bf T}_q(I)={\bf T}_1(I,\tilde{w}_1,\dots,\tilde{w}_{q-1})$. Then 
\begin{equation}\label{lunga}
\begin{split}
\beta_q(I)&=\underset{\{w_1,\dots,w_{q-1}\}}\genv{\bf T}_1(I,w_1,\dots,w_{q-1})\\ &\leq \underset{\{w_1,\dots,w_{q-1}\}}\genv{\bf mult}(I,w_1,\dots,w_{q-1}) \\ 
& \leq {\bf mult}(I,\tilde{w}_1,\dots,\tilde{w}_{q-1})\\ &\leq \left({\bf T}_1(I,\tilde{w}_1,\dots,\tilde{w}_{q-1})\right)^{n-q+1}=\left({\bf T}_q(I)\right)^{n-q+1},
\end{split}
\end{equation}
The first and third inequalities in \eqref{lunga} are an application of \eqref{inequalitymult}, and the second inequality is a consequence of the upper semicontinuity of ${\bf mult}$ (\cite[Theorem 2.1]{Da93}). Here we 
think of $(I,w_1,\dots,w_{q-1})$ as an ideal depending on parameters, namely the coefficients of the 
linear functions $w_1,\dots,w_{q-1}$. By Theorem \ref{nico}, we now conclude that $$D_q(I)=\beta_q(I)\leq \left({\bf T}_q(I)\right)^{n-q+1}.$$
\end{proof}

\section{Acknowledgements}
This research was partially supported by NSF Grant DMS 13-61001 of John D'Angelo. The author would like to thank Professor D'Angelo for his support and guidance. 
The author also wishes to thank the anonymous referees for helpful suggestions. During the revision process of this paper, the author received from Nicoara the preprint \cite{BN17}, which acknowledges this work and corrects \cite{BN15}.

\end{document}